\documentclass{article}

\usepackage{graphicx}

\usepackage{authblk}

\usepackage{tikz}
\usetikzlibrary{arrows,calc}

\tikzstyle{arc}=[->,shorten <=3pt, shorten >=3pt,
                 >=stealth, line width=1.1pt]
\tikzstyle{edge}=[shorten <=2pt, shorten >=2pt,
                  >=stealth, line width=1.1pt]
\tikzstyle{vertex}=[circle, fill=white, draw,
                    minimum size=5pt,
                    inner sep=0pt, outer sep=0pt]

\usepackage{float}

\usepackage{xcolor}

\usepackage{amsmath}
\usepackage{amssymb}

\usepackage{amsthm}

\usepackage[capitalize]{cleveref}

\newtheorem{theorem}{Theorem}
\newtheorem{lemma}[theorem]{Lemma}
\newtheorem{corollary}[theorem]{Corollary}
\newtheorem{proposition}[theorem]{Proposition}

\theoremstyle{definition}
\newtheorem{problem}[theorem]{Problem}
\newtheorem{question}[theorem]{Question}

\title{Critical Kernel Imperfectness in $4$-quasi-transitive and
$4$-anti-transitive digraphs of small diameter%
\thanks{The authors gratefully acknowledge support from grants SEP-CONACYT
A1-S-8397, DGAPA-PAPIIT IA101423, and CONACYT FORDECYT-PRONACES/39570/2020.
The first author is supported by postdoctoral grant “Estancias Posdoctorales por
México” by CONACYT (CVU: 622815)}}

\author[1]{Germ\'an~Ben\'itez-Bobadilla\thanks{german@ciencias.unam.mx}}
\author[2]{Hortensia~Galeana-S\'anchez\thanks{hgaleana@matem.unam.mx}}
\author[1]{C\'esar~Hern\'andez-Cruz\thanks{chc@ciencias.unam.mx}}

\affil[1]{Facultad de Ciencias, Universidad Nacional Aut\'onoma de M\'exico}
\affil[2]{Instituto de Matem\'aticas, Universidad Nacional Aut\'onoma de
          M\'exico}

\begin{document}
\date{}

\maketitle
\begin{abstract}
    A kernel in a digraph is an independent and absorbent subset of its vertex
    set.   A digraph is critical kernel imperfect if it does not have a kernel,
    but every proper induced subdigraph does.   In this article, we characterize
    asymmetrical $4$-quasi-transitive and $4$-transitive digraphs, as well as
    $2$-anti-transitive, and asymmetrical $4$-anti-transitive digraphs with
    bounded diameter, which are critical kernel imperfect.
\end{abstract}

\section{Introduction}

All digraphs considered are finite, with neither multiple arcs nor loops. For
general concepts we refer the reader to \cite{bang2018,bang2009}. 

Let $D$ be a digraph.  A subset $S$ of $V(D)$ is \textit{absorbent} if for every
vertex $u$ not in $S$ there is a vertex $v \in S$ such that $(u,v) \in A(D)$. A
\textit{kernel} is a subset $K$ of $V(D)$ which is independent and absorbent.
Kernels in digraphs were introduced by von Neumann and Morgenstern in the
context of game theory \cite{vonNeumann1944}, and received a lot of attention at
the end of the twentieth century for their relation to the Strong Perfect Graph
Conjecture (see e.g. \cite{borosDAM306}), now the Strong Perfect Graph Theorem
\cite{chudnovskyAM164}.   Although kernels are useful to model a wide variety of
problems, it is computationally hard to determine whether a digraph has a
kernel. Chv\'atal proved that the recognition of digraphs that have a kernel is
an NP-complete problem \cite{chvatal1973}. Furthermore, Hell and Hernández-Cruz
proved that this problem remains NP-complete even when the underlying graph is
3-colourable \cite{hellDMGT34} (this in contrast to the fact that every
bipartite digraph has a kernel).  For this reason, it is interesting to find
sufficient conditions, which can be efficiently verified, implying the existence
of a kernel in a digraph.

A digraph such that every proper induced subdigraph has a kernel is
\textit{kernel-perfect} if it has a kernel, and \textit{critical kernel
imperfect} (\textit{CKI} for short) if not.   As classical examples, we have
that directed odd cycles are critical kernel imperfect digraphs, and directed
even cycles are kernel-perfect digraphs.  It is clear from these definitions
that a digraph is kernel-perfect if and only if it does not contain a critical
kernel imperfect digraph as an induced subdigraph.  For this reason, for a
family of digraphs $\mathcal{D}$, the knowledge of the exact digraphs in
$\mathcal{D}$ which are critical kernel imperfect gives us a sufficient
condition for a graph in $\mathcal{D}$ to have a kernel.   Moreover, if the
number of critical kernel imperfect digraphs in $\mathcal{D}$ is finite, a
polynomial time algorithm for recognizing kernel perfect members of
$\mathcal{D}$ trivially exists. Therefore, it is useful to exhibit the exact set
of critical kernel imperfect digraphs in a family of digraphs.  This has been
done for some of the best known families of digraphs, which we now quickly
survey.

For an integer $m$, with $m \ge 2$, and a nonempty subset $J$ of $\mathbb{Z}_m -
\{ 0 \}$, the \textit{circulant digraph} $\overrightarrow{C}_m(J)$ is defined to
have vertex set $V(\overrightarrow{C}_m(J))=\mathbb{Z}_m$ and arc set
$A(\overrightarrow{C}_m(J)) = \{ (i,j) \colon\ i, j \in \mathbb{Z}_m, j-i \in J
\}$.  There are two cases of circulant digraphs of special interest to us.   Let
$n$ be an integer, with $n \ge 3$. The directed cycle $\overrightarrow{C}_n$ is
the circulant digraph $\overrightarrow{C}_n(\{ 1 \})$. Clearly, the complement
of a circulant digraph is also a circulant digraph; complements of directed
cycles will be relevant for this subject.  The \textit{directed $n$-antihole}
$\overrightarrow{A}_n$ is the circulant digraph $\overrightarrow{C}_n(J)$, where
$J = \{ 1, \dots, n-2 \}$; note that $\overrightarrow{A}_3 =
\overrightarrow{C}_3$. A \textit{directed antihole} is an $n$-directed antihole
for some integer $n$. The family of directed antiholes is denoted by
$\mathcal{A}$.

Regarding small digraphs, in \cite{balbuenaAKCE11}, Balbuena, Guevara, and Olsen
characterized the asymmetrical critical kernel imperfect digraphs with at most 7
vertices as follows.

\begin{theorem}\cite{balbuenaAKCE11}
\label{thm:ckiless7}
  There are exactly four asymmetrical critical kernel imperfect digraphs of
  order less than $8$, namely, $\overrightarrow{C_3}$, $\overrightarrow{C_5}$,
  $\overrightarrow{C_7}$, and $\overrightarrow{C_7}(1,2)$.
\end{theorem}

An arc $(u,v)$ of a digraph $D$ is \textit{symmetrical} if $(v,u)$ is also an
arc of $D$, and \textit{asymmetrical} otherwise. If every arc of $D$ is
asymmetrical we will say that $D$ is an \textit{asymmetrical digraph}.
Analogously, if every arc of $D$ is symmetrical we will say that $D$ is an
\textit{symmetrical digraph}. A $k$-path ($k$-cycle) is a directed path (cycle)
of length $k$.

As usual, for a nonempty subset $S$ of $V(D)$, we denote by $D[S]$ the
subdigraph of $D$ induced by $S$. For digraphs $D$ and $H$, we say that $D$ is
\textit{$H$-free} if $D$ does not contain $H$ as an induced subdigraph.   If
$\mathcal{F}$ is a set of digraphs, we say that a digraph $D$ is
$\mathcal{F}$-free if it is $F$-free for every $F \in \mathcal{F}$.

In \cite{galeana86DM}, Galeana-Sánchez and Neumann-Lara characterized the
semicomplete critical kernel imperfect digraphs as follows.

\begin{theorem}\cite{galeana86DM}
\label{thm:semicompCKI}
  A semicomplete digraph is critical kernel imperfect if and only if it is a
  directed antihole.   Therefore, a semicomplete digraph is kernel-perfect if
  and only if it is $\overrightarrow{\mathcal{A}}$-free.
\end{theorem}

A natural step after semicomplete digraphs, is to consider generalizations of
tournaments (or more precisely, generalizations of semicomplete digraphs). One
of the better known such generalizations, is the family of locally-semicomplete
digraphs.   A digraph is \textit{locally semicomplete} if the neighbourhood of
any vertex induces a semicomplete digraph.   Another two interesting families
arise when we only ask for either the in-neighbours or the out-neighbours of
each vertex to induce a semicomplete digraph.   A digraph $D$ is \textit{locally
(out-)in-semicomplete} if the (out-)in-neighbourhood of every vertex of $D$
induces a semicomplete digrasph. The critical kernel imperfect digraphs which
are locally in-semicomplete, locally out-semicomplete or locally semicomplete
digraphs, were characterized by Galeana-Sánchez and Olsen in
\cite{galeanaDMTCS18}.

\begin{theorem}\cite{galeanaDMTCS18}
  \label{thm:ckilocalsemi}
  Let $D$ be a digraph.   If $D$ is locally semicomplete, locally
  in-semicomplete, or locally out-semicomplete, then $D$ is critical kernel
  imperfect if and only if it is a directed odd cycle, a directed antihole, or
  $D \cong \overrightarrow{C}_7(1, 2)$.
\end{theorem}

Arc-locally semicomplete digraphs were introduced by Bang-Jensen in
\cite{bang1993}, as a common generalization of semicomplete and semicomplete
bipartite digraphs. Critical kernel imperfect digraphs which are arc-locally
semicomplete where characterized by Galeana-S\'anchez in \cite{galeanaDM306}.

\begin{theorem}\cite{galeanaDM306}
    Let $D$ be a digraph.   If $D$ is arc-locally semicomplete, then it is
    critical kernel imperfect if and only if it is a directed odd cycle or a
    directed antihole.
\end{theorem}

A digraph is \emph{arc-locally in-semicomplete} (\emph{arc-locally
out-semicomplete}) if for any pair of adjacent vertices $x$ and $y$, every
in-neighbour (out-neighbour) $z$ of $x$ is adjacent to every in-neighbour
(out-neighbour) $w$ of $y$, or $w = z$.   Wang characterized in
\cite{wangGCOM35} critical kernel imperfect digraphs which are arc-locally
(out-)in-semicomplete.

\begin{theorem}
\cite{wangGCOM35}    
    Let $D$ be an arc-locally (out-)in-semicomplete digraph.   If $D$ is a
    critical kernel imperfect digraph, then $D$ is a directed odd cycle or a
    directed antihole.
\end{theorem}

Let $k$ be an integer, with $k \ge 2$. A digraph $D$ is
\textit{$k$-quasi-transitive} if for every pair of vertices $u$ and $v$ of $D$,
the existence of a $k$-path from $u$ to $v$ in $D$ implies that there is an arc
between $u$ and $v$. A \textit{quasi-transitive digraph} is a
$2$-quasi-transitive digraph. A digraph $D$ is \textit{$k$-transitive} if for
every pair of vertices $u$ and $v$ of $D$, the existence of a $k$-path from $u$
to $v$ in $D$ implies $(u,v)$ is an arc in $D$. A \textit{transitive digraph} is
a $2$-transitive digraph. Trivially, every $k$-transitive digraph is a
$k$-quasi-transitive digraph. In particular, every transitive digraph is a
quasi-transitive digraph. Moreover, every semicomplete digraph is a
quasi-transitive digraph.

Quasi-transitive digraphs are also traditionally considered among the
generalizations of semicomplete digraphs. The family of quasi-transitive
digraphs was introduced by Ghouila-Houri in \cite{Ghouila1962} to characterize
comparability graphs as those graphs that admit a quasi-transitive orientation.
In \cite{bang1995JGT}, Bang-Jansen and Huang gave a recursive characterization
of quasi-transitive digraphs, which has led to solutions of many (usually
difficult) problems. In \cite{hernandezDM312}, Hernández-Cruz and
Galeana-Sánchez introduced $k$-transitive and $k$-quasi-transitive digraphs,
generalizing the notions of transitive and quasi-transitive digraphs,
respectively. The properties of quasi-transitive digraphs and their
generalizations have been studied by several authors. In \cite{galeana2018}, a
thorough compendium dedicated to quasi-transitive digraphs and their extensions
can be found.

In \cite{galeanaDMTCS18}, Galeana-S\'anchez and Olsen analyzed those digraphs
having an underlying perfect graph.   Since underlying graphs of
quasi-transitive digraphs are comparability digraphs, the following result
completely characterizes the critical kernel imperfect quasi-transitive
digraphs.

\begin{theorem}\cite{galeanaDMTCS18} Let $D$ be a digraph.  If the underlying
  graph of $D$ is a perfect graph, then $D$ is critical kernel imperfect if and
  only if $D$ is a directed antihole.
\end{theorem}

Also in \cite{galeanaDMTCS18}, Galeana-S\'anchez and Olsen studied asymmetric
critical kernel imperfect $3$-quasi-transitive digraphs.   The analysis for the
general case (not necessarily asymmetric) was finished by Wang in
\cite{wangGCOM35}.

\begin{theorem}
\cite{wangGCOM35}    
    Let $D$ be a $3$-quasi-transitive digraph.   If $D$ is a critical kernel
    imperfect digraph, then $D$ is a directed antihole.
\end{theorem}

Recently, $k$-anti-transitive digraphs have gained popularity in the context of
Seymour's Second Neighbourhood Conjecture.   A digraph $D$ is
\textit{$k$-anti-transitive} if the existence of a $k$-path from vertex $u$ to
vertex $v$ implies that $(u,v)$ is not an arc of $D$.   Notice that
$2$-anti-transitive digraphs arise naturally as those digraphs not having a
transitive tournament on $3$ vertices as a subgraph.   Seymour's Second
Neighbourhood Conjecture has been proved for $2$-anti-transitive digraphs
\cite{brantnerI2}, $3$- and $4$-anti-transitive digraphs \cite{daamouchDAM304},
$5$-anti-transitive oriented graphs \cite{daamouchDAM285}, and
$6$-anti-transitive oriented graphs \cite{hassanDAM292}.   It is surprinsing
that these nice results have been proved, even in the abscence of a description
of the global structure of $k$-anti-transitive digraphs.   In the present work,
we characterize critical kernel imperfect digraphs for some subclasses of $2$-
and $4$-anti-transitive digraphs, providing further evidence that these families
may have a rich structure which is worth studying.

In this work, we characterize critical kernel imperfect digraphs in some
well-stablished families of digraphs.   In the context of $k$-quasi-transitive
digraphs, we prove that there are exactly two asymmetrical $4$-quasi-transitive
digraphs, namely $\overrightarrow{A_3}$ and $\overrightarrow{C_5}$. Afterwards,
the general case for $4$-transitive digraphs is considered, and we show that the
only $4$-transitive critical kernel imperfect digraphs are
$\overrightarrow{A_3}$ and $\overrightarrow{A_4}$. Finally, we consider
anti-transitive digraphs, showing that there are unique critical kernel
imperfect digraphs for the families of the $2$-anti-transitive digraphs of
diameter $2$, and the asymmetrical $4$-anti-transitive digraphs of diameter $3$.

The rest of the article is organized as follows.   The remaining of this section
is devoted to introduce additional notation and state some basic known results
on kernel-perfect and critical kernel imperfect digraphs.   In Section
\ref{sec:4qt}, a characterization of asymmetric critical kernel imperfect
$4$-quasi-transitive digraphs is presented.   Critical kernel imperfect
$4$-transitive digraphs are characterized in Section \ref{sec:4t}.
Antitransitive digraphs are considered in Section \ref{sec:anti}; critical
kernel imperfect $2$-anti-transitive and $4$-anti-transitive with diameter $3$
are characterized.   Conclusions and open problems are presented in the final
section.

If $X$ and $Y$ are disjoint subsets of $V(D)$, then $X \to_D Y$ means that
$(x,y)$ is an arc of $D$ for every vertex $x$ of $X$ and every vertex $y$ of
$Y$. If $X  \to_D Y$ and $Y  \to_D X$, then we write $X \leftrightarrow_D Y$.
Whenever $X=\{ x \}$ or $Y=\{ y \}$, we write $x \to_D Y$ or $X  \to_D y$
instead of $\{ x \}  \to_D Y$ or $X  \to_D \{ y \}$, respectively. If $X  \to_D
Y$ and there is no arc from $Y$ to $X$, then we write $X \mapsto_D Y$. When
there is no ambiguity, we will omit the subscript $D$.

It is clear from the definition that critical kernel imperfect digraphs are
minimal digraphs without a kernel, this is, every digraph without a kernel
contains a critical kernel imperfect digraph. From this observation, the
following proposition, which will be used extensively in the remaining of this
work without explicit mention, is easily obtained.

\begin{proposition}
If $D$ is a critical kernel imperfect or a kernel-perfect digraph, then $D$ does
not properly contain a critical kernel imperfect subdigraph.
\end{proposition}

To finish this section, we recall a result of Berge and Duchet from
\cite{BergeDuchet}, which will also be used many times without explicit mention.

\begin{theorem}\cite{BergeDuchet} If $D$ is a critical kernel imperfect digraph,
  then $D$ is strong.
\end{theorem}


\section{$4$-quasi-transitive digraphs}
\label{sec:4qt}

The following results describe the structure of $k$-quasi-transitive digraphs
with diameter at least $k + 2$. In the first of the results, the case when $k$
is even, is due to Wang and Zhang \cite{wangDM16}, and the case when $k$ is odd
is due to Alva-Samos and Hernández-Cruz \cite{alvaJGT96}.

\begin{theorem}\cite{alvaJGT96, wangDM16}
\label{thm:kqtlargediam}
  Let $k \ge 3$ be an integer and let $D$ be a $k$-quasi-transitive digraph.
  Suppose that $P = (x_0, \dots, x_r)$ is a shortest path from $x_0$ to $x_r$
  in $D$, with $r \ge k + 2$.
  \begin{enumerate}
    \item If $k$ is even, then $D[V(P)]$ is a semicomplete digraph and $x_j \to
      x_i$ for $1 \le i + 1 < j \le r$.

    \item If $k$ is odd, then $D[V(P)]$ is either a semicomplete digraph and
      $x_j \to x_i$ for $1 \le i + 1 < j \le r$, or $D[V (P)]$ is a semicomplete
      bipartite digraph and $x_j \to x_i$ for $1 \le i + 1 < j \le r$ and $i
      \not\equiv j$ (mod $2$).
  \end{enumerate}
\end{theorem}

\begin{lemma}\cite{wangDM16}
\label{lem:kqtwang}
  Let $k$ be an even integer with $k\ge 4$, and let $D$ be a strong
  $k$-quasi-transitive digraph. Suppose that $P=(v_0,v_1,\dots,v_{k+2})$ is a
  shortest path from $v_0$ to $v_{k+2}$ in $D$. For any $v \in V (D)-V (P)$, if
  $(v,V (P)) \ne \varnothing$ and $(V(P), v) \ne \varnothing$, then either $v$
  is adjacent to every vertex of $V(P)$, or $\{ v_{k+2}, v_{k+1}, v_k,
  v_{k-1}\}\mapsto v \mapsto \{v_0, v_1, v_2, v_3\}$. In particular, if $k = 4$,
  then $v$ is adjacent to every vertex of $V(P)$.
\end{lemma}

\begin{theorem}\cite{wangDM16}
\label{thm:4qtsemicomp}
  Let $k$ be an even integer with $k \ge 4$, and let $D$ be a strong
  $k$-quasi-transitive digraph. If $P = (v_0, \dots, v_{k+2})$ is a shortest
  path from $v_0$ to $v_{k+2}$, then the subdigraph induced by $V(D)-V(P)$ is a
  semicomplete digraph.
\end{theorem}

It follows from \Cref{thm:kqtlargediam,lem:kqtwang,thm:4qtsemicomp} that a
$4$-quasi-transitive digraph of diameter at least $6$ is a semicomplete digraph.
Also, by \Cref{thm:semicompCKI}, the unique semicomplete critical kernel
imperfect digraphs are directed antiholes, which have diameter $2$. Hence, the
following result follows.

\begin{corollary}
\label{cor:4qtdiam6}
  There are no critical kernel imperfect $4$-quasi-transitive digraphs of
  diameter at least 6.
\end{corollary}

In view of \Cref{cor:4qtdiam6}, and since digraphs with diameter $1$ are
complete, and thus kernel-perfect, we will focus on $4$-quasi-transitive
digraphs of diameter $k$, with $2 \le k \le 5$.

\begin{lemma}
\label{lem:4qtdiam2}
  The unique asymmetrical critical kernel imperfect digraph of diameter 2 is
  $\overrightarrow{A}_3$.
\end{lemma}

\begin{proof}
Let $D$ be an asymmetrical critical kernel imperfect of diameter $2$, and let
$x$ be a vertex of $D$. Consider $S_1 = \{ v \in V(D) \colon\ d(v,x) = 1 \}$ and
$S_2 = \{ v \in V(D) \colon\ d(v,x) = 2 \}$. Since $D$ does not have a kernel,
$S_2$ is nonempty.  Also, since $D$ is strong, it follows that $N^+(x) \ne
\varnothing$. Moreover, since $D$ is asymmetrical, then $N^+(x) \subseteq S_2$.
Thus, for $w \in N^+(x)$, there is $u \in S_1$ such that $(w,u,x)$ is a path in
$D$. Thus, $(w,u,x,w)$ is a $3$-cycle in $D$, and since $D$ is asymmetrical,
then $D[\{ w,u,x,w\}] \cong \overrightarrow{A}_3$. Since $D$ is critical kernel
imperfect, it follows that $D \cong \overrightarrow{A}_3$.
\end{proof}

\begin{corollary}
Let $D$ be an asymmetrical digraph of diameter 2. If $D$ has no induced copy of
$\overrightarrow{C}_3$ as a subdigraph, then $D$ has a kernel. 
\end{corollary}

\begin{proposition}
\label{pro:cki4qt2symarc}
There is no $4$-quasi-transitive critical kernel imperfect of diameter 2 such
that every $3$-cycle has at least 2 symmetrical arcs.
\end{proposition}

\begin{proof}
Proceeding by contradiction, suppose that there is $D$ a critical kernel
imperfect $4$-quasi-transitive digraph of diameter $2$ such that every $3$-cycle
has at least $2$ symmetrical arcs.

Let $v \in V(D)$ and $D_1 = D-v$. Since $D$ is a critical kernel imperfect, then
$D_1$ has a kernel $N'$. Observe that $N'\cup \{ v\}$ is not an independent set,
otherwise $N'\cup \{ v\}$ would be a kernel of $D$. Moreover, if there is an arc
from $v$ to some vertex of $N'$, then $N'$ is a kernel of $D$, which is
impossible. Thus, we can assume that there is a vertex $z$ of $N'$ such that
$(z,v)\in A(D)$ and there is no arc from $v$ to some vertex in $N'$.

Let $S_1$ and $S_2$ be defined as $S_1 = \{ z \in N' \colon\ d(z,v) = 1 \}$ and
$S_2 = N'-S_1$. By definition $S_2 \cup \{v\}$ is an independent set. Consider,
$B = \{ u \in V(D) \colon\ u  \to_D z, u  \to_{D^c} v \text{ and } u \to_{D^c}
x, \text{ for some } z \in S_1 \text{ and for all } x \in S_2 \}$, note that $B
\ne \varnothing$, otherwise $S_2 \cup \{v\}$ is a kernel of $D$. Moreover, let
$S_0$ be the subset of $S_1$ such that $S_0 = \{ z \in S_1 \colon\ \text{there
is } u \in B \text{ such that } u \to_D z \}$, observe that $S_0 \ne
\varnothing$.

\textbf{Claim 1}. If $v \to_D w$ and there is $z \in S_1$ such that $w \to_D z$,
then $v \leftrightarrow_D w$ and $w \leftrightarrow_D z$.

\noindent \textit{Proof of Claim 1.} Let $w \in N^+(v)$ and $z \in S_1$ such
that $w \to_D z$. It follows that $(v,w,z,v)$ is a 3-cycle of $D$. Even more, by
hypothesis, and since $z \mapsto v$, then $v \leftrightarrow_D w$ and $w
\leftrightarrow_D z$. This ends the proof of Claim 1.

As a consequence of Claim 1, and by definition of $B$, we have that if $u \in
B$, then $\{v,u\}$ is an independent set.

\textbf{Claim 2}. $S_2=\varnothing$.

\noindent \textit{Proof of Claim 2.} Suppose, for the sake of contradiction,
that there is $x \in S_2$. Since the diameter of $D$ is $2$, then there is $y
\in V(D)$ such that $(v,y,x)$ is a path in $D$. As $N'$ is a kernel and by
definition of $B$, it follows that $y \notin S_1$ and $y \notin B$. Let $u \in
B$.  It follows that there is $z \in S_0$ such that $u \to_D z$.  Moreover,
$(u,z,v,y,x)$ is a $4$-path in $D$. Since $D$ is $4$-quasi-transitive, then
there is an arc between $u$ and $x$. By definition of $B$, $(u,x) \notin A(D)$,
thus $x \to_D u$. Now, since the diameter of $D$ is 2, then there is $w \in
V(D)$ such that $(v,w,z)$ is a path in $D$. By Claim 1, we have that $v
\leftrightarrow_D w$ and $w \leftrightarrow_D z$. Hence, $(x,u,z,w,v)$ is a
$4$-path in $D$, and by hypothesis there is an arc between $x$ and $v$, which is
impossible by definition of $S_2$. Therefore $S_2=\varnothing$. This ends the
proof of Claim 2.

\textbf{Claim 3}. If $x \in V(D) \setminus(\{ v\} \cup S_0 \cup B)$, then $x
\to_D v$.

\noindent \textit{Proof of Claim 3.} Let $x \in V(D) \setminus (\{ v\} \cup S_0
\cup B)$. On the one hand, if $x \in N'$, then, by Claim 2, we have $x \in S_1
\setminus S_0$. Thus $x \to_D v$. On the other hand, if $x \in V(D) \setminus
N'$, then there is $y \in N'$ such that $x  \to_D y$. By Claim 2, we conclude $y
\in S_1$. Since $x \notin B$, it follows that $x \to_D v$. This ends the proof
of Claim 3.

\textbf{Claim 4}. $|S_0|=1$.

\noindent \textit{Proof of Claim 4.} Observe that $S_0 \ne \varnothing$.
Proceeding by contradiction suppose that $|S_0| \ge 2$. Let $z_1,z_2 \in S_0$.
By definition of $S_1$, $z_i \mapsto v$ with $i \in \{1,2\}$. Since the diameter
of $D$ is 2, then there is $w_i \in V(D)$ such that $(v,w_i,z_i)$ is a path in
$D$, for each $i \in \{1,2\}$. In addition, by Claim 1, we have that $v
\leftrightarrow_D w_i$ and $w_i \leftrightarrow_D z_i$. If $w_1 \ne w_2$, then
$(z_1,w_1,v,w_2,z_2)$ is a $4$-path in $D$. Since $D$ is $4$-quasi-transitive,
then there is an arc between $z_1$ and $z_2$ which is impossible because $N'$ is
an independent set. Hence, we can assume that $w_1=w_2$. As $z \in S_0$, then
there is $u \in B$ such that $u \to_D z_1$. It follows that $(u,z_1,w,z_2,v)$ is
a $4$-path in $D$. Since $D$ is a $4$-quasi-transitive digraph, then there is an
arc between $u$ and $v$, which is a contradiction. Therefore, $|S_0|=1$. This
ends the proof of Claim 4. 

\textbf{Claim 5}. $|B|=1$.

\noindent \textit{Proof of Claim 5.}  Suppose, for the sake of contradiction,
that there are $u_1,u_2 \in B$. Since the diameter of $D$ is $2$, then there is
$y_i \in V(D)$ such that $(v,y_i,u_i)$ is a path in $D$, for each $i \in
\{1,2\}$. By definitions of $S_0$ and $B$, then $y_i \notin (S_0 \cup B)$. By
Claim 3, we have that $y_i \leftrightarrow_D v$, for each $i \in \{1,2\}$. It
follows that $(u_2,z,v,y_1,u_1)$ is a $4$-path in $D$, where $z$ is the unique
vertex in $S_0$. Since $D$ is a $4$-quasi-transitive digraph, then there is an
arc between $u_1$ and $u_2$. Suppose without loss of generality that $u_2 \to_D
u_1$. As the diameter of $D$ is 2, there is $w \in V(D)$ such that $(v,w,z)$ is
a path in $D$, by Claim 1, $v \leftrightarrow_D w$ and $w \leftrightarrow_D z$.
We have that $(u_2,u_1,z,y,v)$ is a $4$-path in $D$. Again, $D$ is a
$4$-quasi-transitive digraph, so there is an arc between $u_2$ and $v$,
contradicting that $\{v,u_2\}$ is an independent set. Therefore, $|B|=1$. This
ends the proof of Claim 5.

To finish the proof, let $u$ the unique vertex in $B$. We will prove that $N =
\{v,u\}$ is a kernel of $D$, reaching the contradiction we are looking for. By
Claim 1, $\{v,u\}$ is an independent set. Let $x$ be a vertex in $V(D) \setminus
N$. If $S_0 = \{x\}$, then, by definition of $S_0$, $x \to_D v$. If $x \in V(D)
\setminus (\{ v \} \cup S_0 \cup N)$, then, by Claim 3, we have that $x \to_D
v$. Hence, $N$ is an absorbent set, and in consequence, a kernel of $D$,
contradicting that $D$ is a critical kernel imperfect. Therefore, we conclude
the desired result.
\end{proof}

Observe that every directed $n$-antihole, with $n \ge 4$, is a critical kernel
imperfect digraph with a 3-cycle with exactly one symmetrical arc. Therefore,
the hypothesis of \Cref{pro:cki4qt2symarc} is sharp.

\begin{lemma}
\label{lem:4qtdiam3local}
Let $D$ be an asymmetrical $4$-quasi-transitive digraph of diameter 3. If $D$ is
strong and $\overrightarrow{A}_3$-free, then $D$ is locally semicomplete.
\end{lemma}

\begin{proof}
  Let $v \in V(D)$. We will prove that $D$ is locally in-semicomplete and
  locally out-semicomplete. Since $D$ is strong, then $N^+(v) \ne \varnothing$
  and $N^-(v) \ne \varnothing$.

  Let $u \in N^+(v)$. By hypothesis, $1 \le d(u,v)\le 3$. Since $D$ is an
  asymmetrical digraph, then $1<d(u,v)$, moreover, $d(u,v) \ne 2$, otherwise $D$
  has an induced copy of $\overrightarrow{A}_3$ as a subdigraph. Hence, $d(u,v)
  = 3$ for all $u \in N^+(v)$. Let $u_1,u_2 \in N^+(v)$. From above, there is a
  path $T$, from $u_1$ to $v$ with length 3. It follows that $T\cup (v,u_1)$ is
  a $4$-path from $u_1$ to $u_2$. Since $D$ is $4$-quasi-transitive, then there
  is an arc between $u_1$ and $u_2$. Hence $D[N^+(v)]$ is a semicomplete
  digraph. Moreover, since $D$ is asymmetrical and has no $\overrightarrow{A}_3$
  as induced subdigraph, $D[N^+(v)]$ is a transitive tournament. Therefore $D$
  is locally out-semicomplete digraph.

  Now, let $w \in N^-(v)$. Following an argument analogous to the one in the
  previous paragraph, we conclude that $D[N^-(v)]$ is a transitive tournament.
  Therefore, $D$ is locally in-semicomplete digraph.

  Hence, $D$ is a locally semicomplete digraph.
\end{proof}

\begin{lemma}
\label{lem:cki4qtdiam3}
    Let $D$ be an asymmetrical $4$-quasi-transitive digraph.   If $D$ has
    diameter $3$, then it is not critical kernel imperfect.
\end{lemma}

\begin{proof}
  If there is a $4$-quasi-transitive critical kernel imperfect of diameter $3$
  $D$, then, $D$ is strong and has no $\overrightarrow{A}_3$ as induced
  subdigraph. Hence, by \Cref{lem:4qtdiam3local}, $D$ is a locally semicomplete
  digraph. From Theorem \ref{thm:ckilocalsemi}, it follows that $D$ is either a
  directed odd cycle, a directed antihole, or $D \cong \overrightarrow{C}_7(1,
  2)$. But $\overrightarrow{A}_3$ and $\overrightarrow{C}_5$ have diameter
  different from $3$; directed odd cycles with at least $7$ vertices and
  $\overrightarrow{C}_7(1, 2)$ are not $4$-quasi-transitive digraphs, and
  $\overrightarrow{A}_n$ with $n \ge 4$ has symmetrical arcs. Therefore, there
  is no asymmetrical critical kernel imperfect $4$-quasi-transitive of diameter
  $3$.
\end{proof}

The following lemmas describe some structural properties of asymmetrical $\{
\overrightarrow{C}_3, \overrightarrow{C}_5 \}$-free $4$-quasi-transitive
digraphs of diameter $4$ or $5$.

\begin{lemma}
\label{lem:4qt-dia4o5}
  Let $D$ be an asymmetrical $4$-quasi-transitive digraph of diameter $4$ or
  $5$. Let $v$ be a vertex of $D$ and $S_i$ be the set $\{ x \in V(D) \colon\
  d(x,v) = i\}$ with $i \in \{ 1, \dots, \textnormal{diam}(D)\}$. If $D$ is $\{
  \overrightarrow{C}_3, \overrightarrow{C}_5 \}$-free, then the following
  properties hold.
  \begin{enumerate}
    \item $S_1=N^-(v)$.
    \item For every $i \in \{2,\dots,\textnormal{diam}(D)\}$, $S_i \to_{D^c}
        v$.
    \item For every $i \in \{1,\dots,\textnormal{diam}(D)-2\}$ and $j \in \{
        i+2, \dots, \textnormal{diam}(D)\}$,  $S_j \to_{D^c} S_i$.
    \item $S_2 \leftrightarrow_{D^c}v$.
    \item $S_5 \leftrightarrow_{D^c}v$.
    \item $N^+(v)\subseteq S_3\cup S_4$ and $v \to_D S_4$.
  \end{enumerate}
\end{lemma}

\begin{proof}
  Let $D$, $v$ and $S_i$ as in the hypotheses. The first three properties follow
  from the definition of $S_i$.
  \begin{enumerate}
    \item[4.] From 2 we have that $S_2 \to_{D^c} v$. To prove that $v
      \to_{D^c}S_2$, we proceed by contradiction. Suppose that there is $x_2 \in
      S_2$ such that $v \to_{D}x_2$. By definition of $S_2$, there is $x_1 \in
      S_1$ such that $(x_2,x_1,v)$ is a path in $D$. It follows that
      $(v,x_2,x_1,v)$ is an induced $\overrightarrow{C}_3$ in $D$, which is a
      contradiction. Therefore, $v \to_{D^c}S_2$.

    \item[5.] From 2 we have that $S_5 \to_{D^c} v$. To prove that $v
      \to_{D^c}S_5$, we proceed by contradiction. Suppose that there is $x_5 \in
      S_5$ such that $v \to_{D}x_5$. By definition of $S_5$, there is $x_i \in
      S_i$, with $i \in \{1,2,3,4\}$, such that $(x_5,x_4,x_3,x_2,x_1,v)$ is a
      path in $D$. It follows that $(v,x_5,x_4,x_3,x_2)$ is a $4$-path in $D$.
      Since $D$ is $4$-quasi-transitive, then there is an arc between $v$ and
      $x_2$, contradicting $4$. Therefore, $v \to_{D^c}S_5$.

    \item[6.] Since $D$ is asymmetrical with diameter 4 or 5, and from parts 1,
      4 and 5, we have that $N^+(v)\subseteq S_3\cup S_4$. On the other hand, $v
      \to_D S_4$ follows from the definition of $S_4$, part 2 and the asymmetry
      of $D$.
  \end{enumerate}
\end{proof}

\begin{lemma}
\label{lem:4qt-dia4o5-2}
  Let $D$ be an asymmetrical $4$-quasi-transitive digraph of diameter $4$ or
  $5$. Let $v$ be a vertex of $D$ and $S_i$ be the set $\{ x \in V(D) \colon\
  d(x,v) = i \}$ with $i \in \{ 1, \dots, \textnormal{diam}(D)\}$. If $D$ is
  $\overrightarrow{C}_3$-free and $\overrightarrow{C}_5$-free, then the
  following properties hold.
  \begin{enumerate}
    \item $S_4$ is an independent set.
    \item $S_3\cap N^+(v)$, induces a transitive tournament.
    \item $S_3 \setminus N^+(v)$ is an independent set.
    \item $(S_3\cap N^+(v)) \to_D (S_3 \setminus N^+(v))$.
    \item For every $x_4 \in S_4$ and every $x_3 \in S_3$, there is an arc
      between $x_4$ and $x_3$.
    \item If $x_4 \in S_4$ and $x_3,x'_3 \in S_3$ such that $x_4 \to_D x_3$ and
      $x_3 \to_D x'_3$, then $x_4 \to_D x_3$.
    \item $S_3$ is an independent set. Moreover, if $S_3\cap N^+(v) \ne
      \varnothing$, then $|S_3|=1$ and $S_4 \to S_3$.
  \end{enumerate}
\end{lemma}

\begin{proof}
  Let $D$, $v$ and $S_i$ be as in the hypotheses. 
  \begin{enumerate}
    \item Proceeding by contradiction, suppose that $S_4$ is not an independent
      set. Thus, there are vertices $x_4$ and $x'_4$ in $S_4$ such that $x_4
      \to_{D}x'_4$. By definition of $S_4$, there exists $x_i \in S_i$, with $i
      \in \{1, 2, 3\}$, such that $(x'_4, x_3, x_2, x_1, v)$ is a path in $D$.
      Moreover, by part $6$ of Lemma \ref{lem:4qt-dia4o5}, we have $v \to_D
      x_4$. It follows that $(v,x_4,x'_4,x_3,x_2)$ is a $4$-path in $D$. Since
      $D$ is $4$-quasi-transitive, then  there is an arc between $v$ and $x_2$,
      which is a contradiction. Hence $S_4$ is independent.

    \item First, we will prove that there is an arc between any two vertices in
      $S_3 \cap N^+(v)$. Let $u,w \in S_3\cap N^+(v)$. By definition of $S_3$,
      there is $x_i \in S_i$, with $i \in \{1,2\}$, such that $(u,x_2,x_1,v)$ is
      a path in $D$. It follows that $(u,x_2,x_1,v,w)$ is a $4$-path in $D$, and
      hence, there is an arc between $u$ and $w$. Therefore, $S_3\cap N^+(v)$
      induces a semicomplete digraph in $D$, even more, since $D$ is
      asymmetrical an $\overrightarrow{C}_3$-free, we can conclude that $S_3\cap
      N^+(v)$ induces a transitive tournament.

    \item Proceeding by contradiction suppose that there are vertices $u$ and
      $w$ in $S_3 \setminus N^+(v)$ such that $w \to_D u$. By definition of
      $S_3$, there is $x_i \in S_i$, with $i \in \{1,2\}$, such that
      $(u,x_2,x_1,v)$ is a path in $D$. It follows that $(w,u,x_2,x_1,v)$ is a
      $4$-path in $D$. Since $D$ is $4$-quasi-transitive, then there is an arc
      between $w$ and $v$, contradicting part $1$ of Lemma \ref{lem:4qt-dia4o5}
      or the choice of $w$. Therefore $S_3 \setminus N^+(v)$ is independent.

    \item Let $u \in S_3 \setminus N^+(v)$ and $w \in S_3\cap N^+(v)$. First we
      will prove that there is an arc between $u$ and $w$. By definition of
      $S_3$, there is $x_i \in S_i$, with $i \in \{1,2\}$, such that
      $(u,x_2,x_1,v)$ is a path in $D$. It follows that $(u,x_2,x_1,v,w)$ is a
      $4$-path in $D$. Thus, there is an arc between $u$ and $w$ in $D$. If $u
      \to_D w$, then $(u,w,x'_2,x'_1,v)$ is a $4$-path in $D$, for some $x'_i
      \in S_i$ with $i \in \{ 1,2\}$. Since $D$ is a $4$-quasi-transitive
      digraph, we have the existence of an arc between $u$ and $v$ which is
      impossible. Therefore, $w \to_D u$.

    \item Let $x_4 \in S_4$ and $x_3 \in S_3$. On the one hand, by definition of
      $S_3$, there is $x_i \in S_i$, with $i \in \{1,2\}$, such that $(x_3, x_2,
      x_1, v)$ is a path in $D$. On the other hand, by part $6$ of Lemma
      \ref{lem:4qt-dia4o5}, $v \to_D x_4$. It implies that $(x_3, x_2, x_1, v,
      x_4)$ is a $4$-path in $D$, thus there is an arc between $x_4$ and $x_3$.

    \item Let $x_4 \in S_4$ and $x_3,x'_3 \in S_3$ be such that $x_4 \to_D x_3$
      and $x_3 \to_D x'_3$. By part $5$ of this lemma, there is an arc between
      $x'_3$ and $x_4$. If $x'_3 \to_D x_4$, then $(x_4,x_3,x'_3,x_4)$ is an
      induced $\overrightarrow{C}_3$ in $D$, which is impossible. Therefore $x_4
      \to_D x'_3$.

    \item Observe that if $S_3 \cap N^+(v) = \varnothing$, then the result
      follows from part $3$ of this Lemma. Suppose that $S_3\cap N^+(v) \ne
      \varnothing$. Let $\{ v_1,\dots,v_n\}$ the acyclic ordering of the
      vertices of $S_3\cap N^+(v)$. By part $5$, there is an arc for between
      $v_1$ and $x_4$, for every vertex $x_4$ in $S_4$. If there is $z \in S_4$
      such that $v_1 \to_D z$, then, by definition of $S_4$, there is $x_i \in
      S_i$, with $i \in \{1,2,3\}$, such that $(x_4,x_3,x_2,x_1,v)$ is a path in
      $D$. Moreover, since $D$ is asymmetrical, we have that $(v,v_1,z,x_3,x_2)$
      is a $4$-path in $D$. It follows that there is an arc between $v$ and
      $x_2$, which is impossible by part $4$ of Lemma \ref{lem:4qt-dia4o5}.
      Hence, $S_4 \to_D v_1$. We will prove that $S_3=\{ v_1\}$. Proceeding by
      contradiction, suppose that there is $u \in S_3 \setminus \{ v_1\}$. By
      parts $2$ and $4$, we have that $u \in N^+(v_1)$, and by definition of
      $S_3$, there is $x_2 \in S_2$ such that $u \to_D x_2$. It follows that
      $(v,x_4,v_1,u,x_2)$ is a $4$-path in $D$, thus there is an arc between $v$
      and $x_2$, contradicting part $4$ of Lemma \ref{lem:4qt-dia4o5}.
    \end{enumerate}
\end{proof}

\begin{lemma}
\label{lem:cki4qtdiam4}
  The unique asymmetrical $4$-quasi-transitive critical kernel imperfect of
  diameter 4, is $\overrightarrow{C}_5$. Moreover, there is no asymmetrical
  $4$-quasi-transitive critical kernel imperfect digraphs of diameter 5.
\end{lemma}

\begin{proof}
    Proceeding by contradiction, suppose that there is $D$ a
    $4$-quasi-transitive asymmetrical critical kernel imperfect of diameter 4 or
    5, different from $\overrightarrow{C}_5$. We know that $D$ is strong, as
    well as $\overrightarrow{C}_3$-free and $\overrightarrow{C}_5$-free. 

    Let $v$ be a vertex of $D$ such that the set $S_i=\{ x \in V(D) \colon\
    d(x,v)=i\}$ is not empty for each $i \in \{ 1, \dots,
    \textnormal{diam}(D)\}$. Note that the hypotheses of Lemmas
    \ref{lem:4qt-dia4o5} and \ref{lem:4qt-dia4o5-2} hold. We will consider two
    cases.

    \textit{Case 1.} Suppose first that $S_3\cap N^+(v) \ne \varnothing$. By
    Lemma \ref{lem:4qt-dia4o5-2}.7, $|S_3|=1$, so suppose that $S_3=\{w\}$. Note
    that $v \to_D w$. First, we will prove that $w \to S_2$. Let $x_2 \in S_2$.
    By definition of $S_2$ there is $x_1 \in S_1$, such that $(x_2,x_1,v)$ is a
    path in $D$. By Lemmas \ref{lem:4qt-dia4o5}.6 and \ref{lem:4qt-dia4o5-2}.7,
    we have that $v \to_D S_4$ and $S_4 \to w$, and therefore it follows that
    $(x_2,x_1,v, x_4,w)$ is a $4$-path in $D$, with $x_4 \in S_4$. Thus there is
    an arc between $x_2$ and $w$. Note that, if $x_2 \to_D w$, then
    $(x_2,w,x'_2,x'_1,v)$ is a $4$-path in $D$, for some $x'_i \in S_i$ with $i
    \in \{ 1,2\}$. Hence, there is an arc between $x_2$ and $v$, which is
    impossible by Lemma \ref{lem:4qt-dia4o5}.4. Therefore, $w \to S_2$.

    Now, we will prove that, for every $x_1 \in S_1$ and every $x_2 \in S_2$,
    there is an arc between $x_1$ and $x_2$. Let $x_1 \in S_1$ and $x_2 \in
    S_2$. Since $w \to_D S_2$, using Lemma \ref{lem:4qt-dia4o5}.7 we have that
    for some $x_4 \in S_4$, the $4$-path $(x_1,v,x_4,w,x_2)$ exists in $D$ and
    hence, there is an arc between $x_1$ and $x_2$. Furthermore, we will show
    that for every $x_1 \in S_1$, either $x_1 \to_D S_2$ or $S_2 \to_D x_1$.
    Suppose that this is not true, so there are $x_1 \in S_1$ and $x_2,x'_2 \in
    S_2$ such that $x_1 \to_D x_2$ and $x'_2 \to_D x_1$. Since $D$ is
    asymmetrical, by definition of $S_2$ it follows that $(x'_2,x_1,x_2,x'_1,v)$
    is a $4$-path in $D$, for some $x'_1 \in S_1$. Hence, there is an arc
    between $x_2$ and $v$, contradicting Lemma \ref{lem:4qt-dia4o5}.4.

    Consider the sets $B = \{ x \in S_1 \colon\ x \to_D S_2\}$ and $C=\{ x \in
    S_1 \colon\ S_2 \to_D x\}$. Notice that $C \ne \varnothing$, $B \to_D S_2$
    and $S_2 \to_D C$. We now show that, $B \to C$. Let $b \in B$ and $c \in C$.
    By definition of $B$ and $C$, we have that $(b,v,w,x_2,c)$ is a $4$-path
    from $b$ to $c$, for some $x_2 \in S_2$. This implies that there is an arc
    between $b$ and $c$. If $c \to_D b$ then, for some $x_2 \in S_2$, we have
    that $(b,x_2,c,b)$ is an induced $\overrightarrow{C}_3$ in $D$, which is
    impossible. Hence, $b \to_D c$. Now, we prove that $w \leftrightarrow_{D^c}
    C$. By Lemma \ref{lem:4qt-dia4o5}.3 we have that $w \to_{D^c}C$. To prove $C
    \to_{D^c}w$, aiming for a contradiction, we suppose that there is $c \in C$
    such that $c \to_D w$. It follows that there exists $x_2 \in S_2$ such that
    $(c,w,x_2,c)$ is is an induced $\overrightarrow{C}_3$ in $D$, contradicting
    that $D$ is $\overrightarrow{C}_3$-free. Therefore, $w \leftrightarrow_{D^c}
    C$. We affirm that $|C|=1$. Indeed, suppose for the sake of contradiction,
    that there are $c_1, c_2 \in C$. It follows that $(c_1,v,w,x_2,c_2)$ is a
    $4$-path in $D$, hence there is an arc between $c_1$ and $c_2$. Suppose
    without loss of generality that $c_1 \to c_2$, then $(c_1,c_2,v,x_4,w)$ is a
    $4$-path in $D$, for some $x_4 \in S_4$. But this implies that there is an
    arc between $c_1$ and $w$, contradicting that $w \leftrightarrow_{D^c}C$.
    Thus $|C|=1$.

    \textit{Case 1.1} If $D$ has diameter $4$, then we claim that $N=\{w,c\}$ is
    a kernel of $D$, with $c$ the unique element of $C$. We have already proved
    that $w \leftrightarrow_{D^c} c$, and hence, $N$ is independent. To prove
    that $N$ is absorbent, it suffices to observe that $S_4 \to_D w$, $v \to_D
    w$, $S_2 \to_D c$, and $B \to_D c$. Therefore, $N$ is a kernel of $D$ which
    is a contradiction.

    \textit{Case 1.2} Suppose that $D$ has diameter $5$, and let $D_5 = D[S_5]$.
    Since $D$ is critical kernel imperfect, then $D_5$ has a kernel, say $N_5$.
    We claim that $N=\{w,c\} \cup N_5$ is a kernel of $D$. To prove the
    independence, note that, by definition, $N_5$ is an independent set of $D_5$
    and, in consequence, of $D$. We have that $w \to_{D^c} c$, and, by Lemma
    \ref{lem:4qt-dia4o5}.3, $N_5 \to_{D^c}w$ and $N_5 \to_{D^c}c$. Thus we have
    to prove that $w \to_{D^c} N_5$ and $c \leftrightarrow_{D^c}N_5$. Let $x_5
    \in N_5$. If $w \to_D x_5$, then, as a consequence of Lemma
    \ref{lem:4qt-dia4o5-2}.7, for some $x_4 \in S_4$ we have that
    $(w,x_5,x_4,w)$ is an induced $\overrightarrow{C}_3$ in $D$, which is
    impossible. Hence $w \to_{D^c} x_5$. On the other hand, if $c \to_{D} x_5$,
    then $(v,w,x_2,c,x_5)$ is a $4$-path in $D$, for some $x_2 \in S_2$, which
    is not possible. Therefore, $c \to_{D^c} x_5$, and $N$ is independent. To
    prove that $N$ is absorbent, it is sufficient to note that $S_4 \to_D w$, $v
    \to_D w$, $S_2 \to_D c$, $B \to_D c$, and $N_5$ is absorbent in $D_5$.
    Therefore, $N$ is a kernel of $D$, contradicting that $D$ is critical kernel
    imperfect.

    \textit{Case 2.} Now, suppose that $S_3\cap N^+(v)=\varnothing$.   By Lemma
    \ref{lem:4qt-dia4o5-2}.7, $S_3$ is an independent set of $D$. Since $S_3\cap
    N^+(v)=\varnothing$, Lemma \ref{lem:4qt-dia4o5}.6 implies $v
    \leftrightarrow_{D^c} S_3$ and $N^+(v)=S_4$. First, observe that if there
    are $x_2 \in S_2$ and $x_3 \in S_3$ such that $x_2 \to_D x_3$, then
    $(x_2,x_3,x'_2,x_1,v)$ is a $4$-path in $D$, for some $x'_2 \in S_2$.
    Therefore, there is an arc between $x_2$ and $v$, which is impossible by
    Lemma \ref{lem:4qt-dia4o5}.4. Therefore, $S_2 \to_{D^c}S_3$.

    Now, we prove that for every $x_3 \in S_3$ we have that either $x_3 \to_{D}
    S_4$ or $S_4 \to_D x_3$. Let $x_3 \in S_3$ and, in order to reach a
    contradiction, suppose that there are $x_4,x'_4 \in S_4$ such that $x_3
    \to_D x_4$ and $x'_4 \to_D x_3$. It follows that $(v,x'_4,x_3,x_4,x'_3)$ is
    a $4$-path in $D$, for some $x'_3 \in S_3$, and thus there is an arc between
    $v$ and $x'_3$, contradicting that $v \leftrightarrow_{D^c}S_3$. Hence, for
    every $x_3 \in S_3$ we have that either $x_3 \to_{D} S_4$, or $S_4 \to_D
    x_3$. Consider the sets $U=\{ x \in S_3 \colon\ x \to_D S_4\}$ and $W=\{ x
    \in S_3 \colon\ S_4 \to_D x\}$. Notice that $W \ne \varnothing$, $U \to_D
    S_4$, and $S_4 \to_D W$.

    Let $D_2 = D[S_2]$. Since $D$ is a critical kernel imperfect, then $D_2$ has
    a kernel, say $N_2$. We claim that for every $x_3 \in S_3$ there is $n \in
    N_2$ such that $x_3 \to_D n$, and moreover, $W \to_D N_2$. Let $x_3 \in
    S_3$, it follows that there is $x_2 \in S_2$ such that $x_3 \to_D x_2$. If
    $x_2 \notin N_2$, then there is $n \in N_2$ such that $x_2 \to_D n$. It
    follows that $(x_3,x_2,n,x_1,v)$ is a $4$-path, for some $x_1 \in S_1$, and
    hence, there is an arc between $x_3$ and $v$, which is impossible. Thus,
    $x_2 \in N_2$. To prove the second claim, assume that $x_3 \in W$. Let $z
    \in N_2$, it follows that $(z,x_1,v,x_4,x_3)$ is a $4$-path in $D$, for some
    $x_1 \in S_1$ and with $x_4 \in S_4$, thus there is an arc between $z$ and
    $x_3$. Since $S_2 \to_{D^c}S_3$, then $x_3 \to_D z$. Therefore, $W \to_D
    N_2$.

    Finally, we claim that $N=S_4\cup N_2$ is a kernel of $D$. To prove that $N$
    is independent, observe that, on the one hand, $N_2$ is already an
    independent set, and on the other, Lemma \ref{lem:4qt-dia4o5-2}.1 implies
    that $S_4$ is independent.   Also, an application of Lemma
    \ref{lem:4qt-dia4o5}.3 implies $S_4 \to_{D^c}N_2$.   So, it remains to prove
    that $N_2 \to_{D^c} S_4$.   For the sake of contradiction, suppose that
    there are $x_4 \in S_4$ and $x_2 \in N_2$ such that $x_2 \to_D x_4$. This
    implies that $(x_2,x_4,x_3,x_2)$ is an induced $\overrightarrow{C}_3$ in
    $D$, for some $x_3 \in W$, contradicting that $D$ is
    $\overrightarrow{C}_3$-free. Hence $N_2 \to_{D^c} S_4$. We conclude that $N$
    is independent.

    For the absorbance of $N$, we have that $v \to_D S_4$, also $N_2$ is a
    kernel of $D_2$, and for every $x_3 \in S_3$ there is $n \in N_2$ such that
    $x_3 \to_D n$.   When $D$ has diameter 5, we also have that, for every $x_5
    \in S_5$, there is $x_4 \in S_4$ such that $x_5 \to_D x_4$.   So, it remains
    to prove that for every $x_1 \in S_1$ there is $x \in N$ such that $x_1
    \to_D x$. Let $x_1 \in S_1$ and $x \in N_2$. Since $(x_1,v,x_4,w,x)$ is a
    $4$-path in $D$, for some $x_4 \in S_4$, then there is an arc between $x_1$
    and $x$. If $x_1 \to_D x$, then we have the desired result. Suppose that $x
    \to_D x_1$, it follows that $\gamma=(x,x_1,v,x_4,w,x)$ is a 5-cycle of $D$,
    with $x_4 \in S_4$ and $w \in W$. Since $D$ is $\overrightarrow{C}_5$-free,
    then $\gamma$ has at least one diagonal. Observe that $x_4
    \leftrightarrow_{D^c}x$, $x \leftrightarrow_{D^c}v$, $v
    \leftrightarrow_{D^c}w$ and $x_4 \to_{D^c}x_1$, it follows that $x
    \to_{D}x_4$, with $x_4 \in S_4$. Hence $N$ is an absorbent set in $D$. Thus
    $N$ is a kernel of $D$, which is a contradiction.

    From above, we conclude that the unique asymmetrical $4$-quasi-transitive
    critical kernel imperfect digraph of diameter $4$ is $\overrightarrow{C}_5$,
    and there is no asymmetrical $4$-quasi-transitive critical kernel imperfect
    digraphs of diameter $5$.
\end{proof}

\begin{theorem}\label{thm:unique-4qt-cki} The unique asymmetrical
    $4$-quasi-transitive critical kernel imperfect digraphs are
    $\overrightarrow{A}_3$ and $\overrightarrow{C}_5$.
\end{theorem}

\begin{proof}
    The result follows from Corollary \ref{cor:4qtdiam6} and Lemmas
    \ref{lem:4qtdiam2}, \ref{lem:cki4qtdiam3} and \ref{lem:cki4qtdiam4}.
\end{proof}

As a consequence of Theorem \ref{thm:unique-4qt-cki} we have the following
result.

\begin{corollary}
    If $D$ is an asymmetrical $4$-quasi-transitive digraph, then it is
    kernel-perfect if and only if it is $\{ \overrightarrow{A}_3,
    \overrightarrow{C_5} \}$-free.
\end{corollary}

\section{$4$-transitive digraphs}
\label{sec:4t}

In \cite{hernandezDMGT33}, Hernández-Cruz gave a structural characterization
of strong $4$-trasitive digraphs. Note that $4$ is the largest value of $k$ such
that strong $k$-transitive digraphs are characterized. 

\begin{theorem}\cite{hernandezDMGT33}
\label{thm:s4t-char}
  Let $D$ be a strong $4$-transitive digraph. Then exactly one of the following
  possibilities holds
  \begin{enumerate}
    \item $D$ is a complete digraph.
    \item $D$ is a 3-cycle extension.
    \item $D$ has circumference 3, a 3-cycle extension as a spanning subdigraph
      with cyclical partition $\{ V_0, V_1, V_2\}$, at least one symmetrical arc
      exists in $D$ and for every symmetrical arc $(v_i,v_{i+1})\in A(D)$, with
      $v_j \in V_j$ for $j \in \{i, i + 1\}$ (mod 3), $|V_i| = 1$ or $|V_{i+1}|
      = 1$.
    \item $D$ has circumference 3, $UG(D)$ is not 2-edge-connected and ${S_1,
      S_2,\dots , S_r}$ are the vertex sets of the maximal 2-edge connected
      subgraphs of $UG(D)$, with $S_i = \{ u_i\}$ for every $2\le i \le r$ and
      such that $D[S_1]$ has a 3-cycle extension with cyclical partition $\{
      V_0, V_1, V_2\}$ as a spanning subdigraph. A vertex $v_0 \in V_0$ (without
      loss of generality) exists such that $v_0 \leftrightarrow_D u_j$, for
      every $2\le j \le n$. Also $|V_0| = 1$ and $D[S_1]$ has the structure
      described in $1.$ or $2.$, depending on the existence of symmetrical arcs.
    \item A complete biorientation of a 5-cycle.
    \item $D$ is a complete biorientation of the star $K_{1,r}$, $r \ge 3$.
    \item $D$ is a complete biorientation of a tree with diameter 3.
    \item $D$ is a strong digraph of order less than or equal to 4 not included
      in the previous families.
  \end{enumerate}

\end{theorem}

With the characterization in Theorem \ref{thm:s4t-char} we determine the
$4$-transitive critical kernel imperfect digraphs.

\begin{theorem}\label{thm:4tCKI} If $D$ is a $4$-transitive critical kernel
    imperfect digraph, then $D \cong \overrightarrow{A}_3$ or $D\cong
    \overrightarrow{A}_4$.
\end{theorem}

\begin{proof}
Let $D$ be a $4$-transitive critical kernel imperfect digraph. Clearly, $D$
satisfies exactly one of the possibilities of Theorem \ref{thm:s4t-char}. It is
not difficult to see that if a digraphs is in the family $1$, $5$, $6$, or $7$
of Theorem \ref{thm:s4t-char}, then the digraph has a kernel, moreover it is a
kernel-perfect digraph. Hence, $D$ is in family $2$, $3$, $4$, or $8$. Notice
that $\overrightarrow{C_3}$ is in family $2$, and every digraph which is in
family $2$, $3$, or $4$, contains $\overrightarrow{C_3}$ as an induced
subdigraph. Hence, we have that if $D$ is in family $2$, then $D\cong
\overrightarrow{C_3}$, and  $D$ is not in family $3$ or $4$.

Suppose that $D$ is in family 8. By Theorem \ref{thm:ckiless7}, if $D$ is an
asymmetrical digraph, then $D\cong \overrightarrow{C_3}$. Even more, by Theorem
\ref{thm:semicompCKI}, if $D$ is a semicomplete digraph, then $D\cong
\overrightarrow{C_3}$ or $D \cong \overrightarrow{A_4}$. Now assume that $D$ is
not an asymmetrical and $D$ is not semicomplete. Since every digraph with at
most 3 vertices is kernel-perfect or is $\overrightarrow{C_3}$, then $D$ has 4
vertices. Let $\{ u,v,z,w \}$ be the vertex set of $D$. Since it is not
semicomplete, the independence number of $D$ is at least $2$, but it must be
exactly $2$, as otherwise $D$ has a kernel. Suppose without loss of generality
that there is no arc between $u$ and $v$. Since $D$ has no kernel, then $\{
u,v\}$ is not an absorbent set. We can assume that $(w,u)$ and $(w,v)$ are not
arcs of $D$. Moreover, since $\{ u,v\}$ is independent and $D$ is strong, it
follows that $(w,z),(z,u),(z,v)\in A(D)$. We know that $\{z\}$ is not a kernel
of $D$, and thus, $(u,z)$ or $(v,z)$ is not an arc of $D$. Suppose, without loss
of generality, that $(v,z) \notin A(D)$, it follows that $(v,w)\in A(D)$. Since
$D$ is critical kernel imperfect, it has no $\overrightarrow{C_3}$ as induced
subdigraph. Thus $(z,w)\in A(D)$. Note that $(u,w) \notin A(D)$, otherwise $\{
w\}$ is not a kernel of $D$. It follows that $\{ u,w\}$ is a kernel of $D$,
contradicting that $D$ is a critical kernel imperfect. Hence $D$ has to be a
semicomplete or an asymmetrical digraph.
\end{proof}

As a consequence of Theorem \ref{thm:4tCKI}, we have the following result.

\begin{theorem}
  If $D$ is a $4$-transitive digraph, then $D$ is kernel-perfect if and only if
  it is $\{ \overrightarrow{A_3}, \overrightarrow{A_4} \}$-free.
\end{theorem}


\section{Antitransitive digraphs}
\label{sec:anti}

Let $k$ be an integer, $k\ge 2$. A digraph $D$ is $k$-anti-transitive if for
every pair of distinct vertices $u$ and $v$ of $D$, the existence of a $k$-path
from $u$ to $v$ in $D$ implies that $(u,v)$ is not an arc of $D$. We denote by
$TT_3$ the asymmetrical transitive tournament with order 3.

\begin{lemma}
\label{lem:2at-noTT3}
    Let $D$ be a $2$-anti-transitive digraph, and let $v$ be a vertex of $D$.
    The following statements hold
    \begin{enumerate}
        \item $D$ has no $TT_3$ as subdigraph.
        \item $N^+(v)$ and $N^-(v)$ are independent sets in $D$.
    \end{enumerate}
\end{lemma}

\begin{proof}
Let $D$ be a $2$-anti-transitive digraph. Note that if $TT_3$ is a subdigraph of
$D$, then there are vertices $x$ and $y$ such that there is a $2$-path from $x$
to $y$ and $(x,y)$ is an arc of $D$, which is impossible by the definition of
$2$-anti-transitive digraph. For the second item, if $v$ is a vertex of $D$ and
there is an arc between two vertices of $N^+(v)$ ($N^-(v)$), then $TT_3$ is a
subdigraph of $D$, which is also impossible by 1.
\end{proof}

\begin{theorem}
    The unique $2$-anti-transitive critical kernel imperfect of diameter $2$ is
    $\overrightarrow{A}_{3}$.
\end{theorem}

\begin{proof}
Let $D$ be an $2$-anti-transitive critical kernel imperfect of diameter 2.   If
$D$ is isomorphic to $\overrightarrow{A}_{3}$, there is nothing to prove.
Proceeding by contradiction, suppose that $D \ncong \overrightarrow{A}_3$. Being
critical kernel imperfect, $D$ is $\overrightarrow{A}_3$-free.

Let $v$ be a vertex of $D$. Consider the sets $S_i=\{ x \in V(D) \colon\
d(x,v)=i\}$ with $i \in\{ 1,2\}$. By hypothesis $V(D)=\{v\} \cup S_1\cup S_2$.
Moreover, since $D$ has no kernel, then $S_2 \ne \varnothing$. By Lemma
\ref{lem:2at-noTT3} it follows that $S_1=N^-(v)$ is an independent set of $D$.
We claim that $N^+(v)\subseteq S_1$. Let $x$ be a vertex in $ N^+(v)$. If $x
\notin S_1$, then $x \in S_2$. Thus there is $w \in S_1$ such that $(x,w,v)$ is
a $2$-path in $D$, which implies that $\gamma=(x,w,v,x)$ is a $3$-cycle in $D$.
Since $D$ is $\overrightarrow{A}_3$-free, then there is a symmetrical arc in
$\gamma$, it follows that $TT_3$ is a subdigraph of $D$, which is impossible.
Hence, $N^+(v)\subseteq S_1$.

By Lemma \ref{lem:2at-noTT3}, $S_1=N^-(v)$ is an independent set of $D$. Let $u
\in V(D) \setminus S_1$. If $u \in S_2$, then there is $w \in S_1$ such that $u
\to_D w$. On the other hand, $D$ is strong, and since $N^+(v)\subseteq S_1$, if
$u=v$, then is $w \in S_1$ such that $x \to_D w$. Therefore, $S_1$ is an
absorbent set and a kernel of $D$, which is a contradiction.
\end{proof}

\begin{lemma}
\label{lem:asym-4at-neigh}
    Let $D$ be an asymmetrical $\overrightarrow{C}_3$-free, $4$-anti-transitive
    digraph with diameter $3$. If $v \in V(D)$, then the following statements
    hold.
    \begin{enumerate}
        \item $d_D(w,v)=3$, for each $w \in N^+(v)$.
        \item $N^+(v)$ is an independent set.
        \item $N^-(v)$ is an independent set.
    \end{enumerate}
\end{lemma}

\begin{proof}
Let $w$ be in $N^+(v)$. Since $D$ is an asymmetrical digraph with diameter $3$,
then $1<d_D(w,v)\le 3$, moreover if $d_D(w,v)=2$, then $\overrightarrow{C}_{3}$
is an induced subdigraph of $D$, which is impossible. Hence $d_D(w,v)=3$.

For the second statement, let $w$ and $u$ be two vertices of $N^+(v)$. By the
first statement, we have that $d_D(w,v)=3$. Thus there is $(w,z,x,v)$ a path
from $w$ to $v$ in $D$. It implies that $(w,z,x,v,u)$ is a path from $w$ to $u$
in $D$ with length 4. Since $D$ is a $4$-anti-transitive digraph, then $(w,u)$
is not an arc of $D$. Therefore, $N^+(v)$ is an independent set.

For the third statement, let $x$ and $y$ be two vertices of $N^-(v)$. Note that
$v \in N^+(x)$, by the first statement it follows that $d_D(v,x)=3$. Thus there
is $(v,z,u,x)$ a path from $v$ to $x$ in $D$. It follows that $(y,v,z,u,x)$ is a
path from $y$ to $x$ in $D$ with length 4. Since $D$ is a $4$-anti-transitive
digraph, then $(y,x)$ is not an arc of $D$. Therefore, $N^-(v)$ is an
independent set.
\end{proof}

\begin{theorem}
    The unique asymmetrical $4$-anti-transitive critical kernel imperfect
    digraph with diameter $3$ is $\overrightarrow{C}_5$.
\end{theorem}

\begin{proof}
    Let $D$ be an asymmetrical $4$-anti-transitive critical kernel imperfect of
    diameter 3. Looking for a contradiction, suppose that $D \ncong
    \overrightarrow{C}_5$. Since it is critical kernel imperfect, $D$ is
    $\overrightarrow{C}_{2n+1}$-free, for every $n \in \mathbb{N}$.

    Let $v_0$ be a vertex of $D$ such that the sets $S_i=\{ x \in V(D) \colon\
    d(x,v_0)=i\}$ are nonempty for $i \in \{1,2,3\}$. By hypothesis $V(D) = \{
    v_0 \} \cup S_1\cup S_2\cup S_3$. Note that $S_1=N^-(v_0)$ and, by Lemma
    \ref{lem:asym-4at-neigh}, $S_1$ is an independent set of $D$. By definition
    $S_i \to_{D^c} v_0$ with $i \in\{ 2,3\}$. Moreover, there is no arc from
    $v_0$ to $y$, for every $y \in S_2$, otherwise $\overrightarrow{C}_3$ is an
    induced subdigraph of $D$, which is impossible. From above, it follows that
    $N^+(v_0)\subseteq S_3$, moreover, since $D$ is strong, then $N^+(v_0)$ is
    nonempty.

    We claim that $S_1 \leftrightarrow_{D^c}S_3$. By definition, $S_3
    \to_{D^c}S_1$. To prove $S_1 \to_{D^c}S_3$, we proceed by contradiction.
    Suppose that there are $x \in S_1$ and $y$ in $S_3$ such that $(x,y)$ is an
    arc of $D$. Since $y \in S_3$, then there are $u \in S_2$ and $w \in S_1$
    such that $(y,u,w,v_0)$ is a path in $D$. Observe that $w \ne x$, otherwise
    $(y,u,x,y)$ is an induced $\overrightarrow{C}_3$ in $D$, which is
    impossible. It follows that $(x,y,u,w,v_0)$ is a path in $D$ with length 4.
    Since $D$ is $4$-anti-transitive, we have that $(x,v_0)$ is not an arc of
    $D$, which is a contradiction. Therefore, $S_1 \leftrightarrow_{D^c}S_3$.

    It follows from Lemma \ref{lem:asym-4at-neigh} that $N^+(v_0)\subseteq S_3$.
    We now prove that $(S_3 \setminus N^+(v_0))\leftrightarrow_{D^c}N^+(v_0)$.
    Let $x \in N^+(v_0)$ and $y \in S_3 \setminus N^+(v_0)$. By definition,
    there are $u \in S_2$ and $w \in S_1$ such that $(y,u,w,v_0)$ is a path in
    $D$ and, even more $(y,u,w,v_0,x)$ is a path in $D$ of length $4$. Since $D$
    is $4$-anti-transitive, we have that $(y,x)$ is not an arc in $D$. To prove
    that $(x,y)$ is not an arc of $D$ suppose, for the sake of contradiction,
    that $(x,y)\in A(D)$. By definition of $S_3$, there are $u'\in S_2$ and
    $w'\in S_1$ such that $(y,u',w',v_0)$ is a path in $D$. It follows that
    $\gamma=(v_0, x,y,u',w',v_0)$ is a $5$-cycle in $D$. We know that $S_1
    \leftrightarrow_{D^c}S_3$, thus $x \leftrightarrow_{D^c}w'$ and $y
    \leftrightarrow_{D^c}w'$. Since $S_2 \leftrightarrow_{D^c}v_0$, then
    $u'\leftrightarrow_{D^c}v_0$. By definition of $S_3 \setminus N^+(v_0)$, we
    also have $v_0 \leftrightarrow_{D^c}y$. In addition, $(u',x)$ is not an arc
    of $D$, because $D$ is $\overrightarrow{C}_3$-free. Even more, $(x,u')$ is
    neither an arc of $D$, because we would have $y,u'\in N^+(x)$, and Lemma
    \ref{lem:asym-4at-neigh} states that $N^+(x)$ is an independent set, but
    $(y,u')$ is an arc of $D$. Thus, $x \leftrightarrow_{D^c}u'$, and we
    conclude that $\gamma$ is an induced $\overrightarrow{C}_5$ in $D$, which is
    a contradiction. Therefore $(S_3 \setminus N^+(v_0)) \leftrightarrow_{D^c}
    N^+(v_0)$.
    
    Let $D_3$ the subdigraph induced by $S_3 \setminus N^+(v_0)$. Since $D$ is
    critical kernel imperfect, $D_3$ has a kernel, say $K_3$. We claim that
    $K=S_1\cup N^+(v_0)\cup K_3$ is a kernel of $D$. Note that $N^+(v_0)\cup
    K_3\subseteq S_3$ and $S_3 \leftrightarrow_{D^c}S_1$, moreover $K_3\subseteq
    (S_3 \setminus N^+(v_0))$ and $(S_3 \setminus N^+(v_0))
    \leftrightarrow_{D^c} N^+(v_0)$. Since $S_1=N^-(v_0)$, by the definition of
    $K_3$, and by Lemma \ref{lem:asym-4at-neigh}, we obtain that $K$ is an
    independent set of $D$. To prove that $K$ is absorbent, let $u \in V(D)
    \setminus K$. If $u=v_0$, then $v_0 \to N^+(x_0)$. If $u \in S_2$, then, by
    definition there is $w \in S_1$, such that $(u,w)$ is an arc of $D$. And, if
    $u \in S_3$, then $u \in ((S_3 \setminus N^+(v_0)) \setminus N_3)$ and, by
    the definition of $K_3$, there is $x \in K_3$ such that $(u,x)$ is an arc of
    $D$. Therefore, $K$ is a kernel of $D$, which is a contradiction. We
    conclude that $\overrightarrow{C}_5$ is the unique $4$-anti-transitive
    critical kernel imperfect with diameter $3$.
\end{proof}


\section{Conclusions}

We were able to characterize all the critical kernel imperfect digraphs which
are $4$-quasi-transitive oriented graphs.   It turns out that there are only two
of them, $\overrightarrow{A_3}$ and $\overrightarrow{C_5}$.  As we mentioned in
the introduction, every semicomplete digraph is $k$-quasi-transitive for every
positive integer $k$ greater than $1$.   Hence, $\overrightarrow{A_n}$ is a
critical kernel imperfect $4$-quasi-transitive digraph for every $n$ at least
$3$.   So, there is an infinite family of critical kernel imperfect
$4$-quasi-transitive digraphs when we admit symmetric arcs.   It is natural to
ask, are the directed antiholes the only such critical kernel imperfect
digraphs?

\begin{problem}
    Characterize all the $4$-quasi-transitive critical kernel imperfect
    digraphs.   In particular, is it true that such digraphs are either
    $\overrightarrow{C_5}$ or a directed antihole?
\end{problem}

In a similar direction, we where able to characterize all the critical kernel
imperfect $4$-transitive digraphs.   It turns out that the complete list only
contains $\overrightarrow{A_3}$ and $\overrightarrow{A_4}$.   Again, it is clear
that $\overrightarrow{A_5}$ is a $5$-transitive digraph which is critical kernel
imperfect.  Could it be the case that all $k$-transitive critical kernel
imperfect digraphs are directed antiholes?

\begin{question}
    Is it true that every $k$-transitive critical kernel imperfect digraph is a
    directed antihole?
\end{question}

The lack of knowledge of the structure of $k$-anti-transitive digraphs prevented
us from getting more general results for the families of $2$- and
$4$-anti-transitive digraphs.   We propose the problem of finding nice
structural descriptions of the families of $k$-anti-transitive digraphs, at
least for some small values of $k$.   Naturally, it would also be nice to know
the complete families of critical kernel imperfect $2$-anti-transitive digraphs.
Notice that $\overrightarrow{C_7}(1,2)$ is a $3$-anti-transitive asymmetric
digraph having diameter $3$, so it is no longer the case that the only critical
kernel imperfect digraphs are directed odd cycles or directed antiholes.

\end{document}